\documentclass[12pt]{iopart}
\usepackage{iopams}
\expandafter\let\csname equation*\endcsname\relax
\expandafter\let\csname endequation*\endcsname\relax
\usepackage{amsmath}
\usepackage{amsthm}

\newtheorem{theorem}{Theorem}[section]
\newtheorem{lemma}[theorem]{Lemma}

\theoremstyle{definition}
\newtheorem{definition}[theorem]{Definition}

\newtheorem{remark}[theorem]{Remark}
\newtheorem{corollary}[theorem]{Corollary}

\begin{document}

\title{General conservation law for a class of physics field theories}

\author{Lauri Kettunen$^1$, Sanna M\"onk\"ol\"a$^1$, Jouni Parkkonen$^2$, Tuomo Rossi$^1$}

\address{$^1$University of Jyväskylä, Faculty of Information Technology, PO Box 35, FI-40014 University of Jyvaskyla, Finland,\vspace*{6pt}\\ 
$^2$University of Jyväskylä, Department of Mathematics and Statistics, PO Box 35, FI-40014 University of Jyvaskyla, Finland}
\ead{lauri.y.o.kettunen@jyu.fi}
\vspace{10pt}
\begin{indented}
\item[]August 2019
\end{indented}

\begin{abstract}
In this paper we form a general conservation law that unifies a class of physics field theories. For this we first introduce the notion of a general field as a formal sum differential forms on a Minkowski manifold. Thereafter, we employ the action principle to define the conservation law for such general fields. By construction, particular field notions of physics, such as electric field strength, stress, strain etc. become instances of the general field. Hence, the differential equations that constitute physics field theories become also instances of the general conservation law. Accordingly, the general field and the general conservation law together correspond to a large class of physics field models. The approach creates solid foundations for multi-physics analysis and is critical in developing software systems for scientific computing; the unifying structure shared by the class of field models makes it possible to implement software systems which are not restricted to finite lists of admissible problems.
\end{abstract}

\vspace{2pc}

%\submitto{\jpa}
%\maketitle

\section{Introduction}

Field theories are central in physics and much of the progress in engineering is achieved by solving corresponding boundary value problems. While the physics field theories are specialized and compartimentalized, the contemporary needs call for a unifying view. In multi-physical software design a coherent view that makes it possible to embed several physics field theories into the same system is a clear advantage in coming up with consistent and concise software systems. 

The aim of this paper is to introduce a general conservation law that covers several branches of physics. For this, instead of starting from any particular field theory, such as electromagnetism or elasticity, we first collect a formal sum of the field types that appear in different theories. Thereafter we present in one token the conservation law for the formal sum exploiting the action principle \cite{Feynman} \cite{Bleecker} \cite{Baez-Muniain}. 

The formal sum and the corresponding conservation law are called the general field and the general conservation law, respectively. They stem from a class of field theories that include in their core i) a pair of fields, ii) a pair of differential equations, one for each field, iii) a constitutive relation between the fields, and iv) a real valued action, which is either minimized, or in a more general sense, which does not change to first order by small variations in the solutions of the differential equations. Such a structure arises from the basic principles shared by many physics field theories in making observations in space-time and in formalizing such observations into algebra. 

When talking about analogies between field theories, elementary physics textbooks already implicitly recognize a common structure behind field theories. However, while analogies may be helpful in building an intuitive understanding of physics, they are not sufficient for the demands of modern needs. Instead, the common structure should be revealed and expressed in proper mathematical terms. 

In this paper we employ differential geometry and exterior algebra to explain the findings and results. While many alternatives exist, this language yields a close link to the existing literature on solution methods of physics boundary value problems. 

The paper is organized as follows. First we will define the notion of a general field in section \ref{sec: general field}. In section \ref{sec:Conservation law} after introducing the basic tools we define the general conservation law in space-time. Thereafter the same is given explicitly in $(1+3)$-dimensions in time and space in theorem \ref{thr:Conservation law in space and time} and corollary \ref{crl:Conservation law in space and time II}.  These theorems are the main results of this paper.  For convenience, the main result is also translated into terms of classical vector analysis in remark \ref{rmk:metric proxy}. In section \ref{sec:Physics models} we derive particular physics field models from corollary \ref{crl:Conservation law in space and time II} to demonstrate the usage of the main result with several examples. The main principles of approximating the general conservation law in finite dimensional spaces for the needs of numerical computing are briefly explained in section \ref{sec:Approximations}. Finally, the conclusions follow in section \ref{sec:conclusions}.

\section{General field as formal sum of $p$-forms}
\label{sec: general field}

To get started we assume space-time is modeled as an $n$-dimensional smooth and oriented Minkowski manifold $\Omega$ equipped with a metric tensor and with the signature $(-,+,+,\dots ,+)$.  Furthermore, at will, we may assume that $\Omega$ is topologically trivial, as the conservation law describes field properties locally in the virtual neighbourhoods of points of $\Omega$, and hence, the conservation law does not depend on the topology. 

The space-like part of $\Omega$ is a Riemannian manifold  \cite{Frankel} \cite{Petersen} denoted by $\Omega_s$, and the boundary of $\Omega_s$ is $\partial\Omega_s$.  A bounded domain in the Euclidean 3-space is a simple and common example of $\Omega_s$.  The time-like part of $\Omega$ is $\Omega_t$. When $\Omega$ is decomposed into space and time, we write $\Omega=\Omega_t\times\Omega_s$. 

Many physics field theories cannot be covered only with ordinary differential forms, and therefore, we introduce also so called $E$-valued and ${\rm End}(E)$-valued forms \cite{Frolicher-Nijenhuis} \cite{Baez-Muniain}. For this, we denote the tangent and cotangent bundles over $\Omega$ by $T\Omega$ and $T^\ast\Omega$. Then, suppose $E$ is a smooth real vector bundle over $\Omega$ and $\bigwedge\nolimits^p T^\ast\Omega$ is a $p$-form bundle over $\Omega$. An $E$-valued differential form of degree $p$ is a section of $E\otimes\bigwedge\nolimits^p T^\ast\Omega$ \cite{Frolicher-Nijenhuis} \cite{Baez-Muniain} \cite{Frankel}. 

Linear maps from a vector space back to itself are mathematically known as endomorphisms. Let $E^*$ be the dual space of $E$.  The endomorphism bundle ${\rm End}(E)$ of $E$ over $\Omega$ is isomorphic to $E\otimes E^*$ \cite[p. 221]{Baez-Muniain}. For the needs of this paper it is sufficient to interpret ${\rm End}(E)$ as $E\otimes E^*$. An ${\rm End}(E)$-valued differential form of degree $p$ is a section of ${\rm End}(E)\otimes\bigwedge\nolimits^p T^\ast\Omega$ \cite{Baez-Muniain}.

To exemplify $E$ and ${\rm End}(E)$-valued forms, let us express them in local coordinates. For this, let $f$ be an ordinary $p$-form, $s$ a section of tangent bundle $E=T\Omega$, and $\Omega$ a simple bounded domain in Euclidean $3$-space. The wedge product between differential forms is denoted by $\wedge$. In this case a section corresponds with a vector field and in local coordinates $s$ can be written as $s=s^x e_x +s^y e_y +s^y e_y$. A $1$, $2$, or $3$-form $f$ can be given by $f=f_x{\rm d}x +f_y{\rm d}y + f_z{\rm d}z$, $f=f_{yz}{\rm d}y\wedge{\rm d}z +f_{zx}{\rm d}z\wedge{\rm d}x+ f_{xy}{\rm d}x\wedge{\rm d}y$, or $f=f_{xyz}{\rm d}x\wedge{\rm d}y\wedge{\rm d}z$, respectively. In this case the $E$-valued $p$-form is a so-called vector valued $p$-form
\begin{displaymath}
s\otimes f=[s^x f, s^yf, s^zf]^T. 
\end{displaymath}
When $s\otimes f$ operates on a $p$-vector (or on a $p$-tuple of vectors $w$), $f(w)$ becomes a real, and the value of the map is vector
\begin{displaymath}
s\otimes f(w) = [f(w)s^x , f(w)s^y, f(w)s^z]^T\in\mathbb{R}^3.
\end{displaymath}
This motivates the name "vector-valued $p$-forms". 

\begin{remark}
Not all tensors are simple, and consequently, all $E$-valued differential forms cannot be given in the form $s\otimes f$. However, they can still be written as a sum of such forms.
\end{remark}

The local expression of $E^*$-valued, that is, of co-vector valued $p$-forms is constructed similarly. Let $s^*$ be a section of $E^*$. In Euclidean spaces $s^*$ is then just a co-vector field. In local coordinates $s^*$ is given by $s^* = s^*_x{\rm d}x + s^*_y{\rm d}y + s^*_z{\rm d}z$. Now, a co-vector valued $p$-form $s^*{\otimes f}$ is
\begin{displaymath}
s^*{\otimes f} = s_x^*\,{\rm d}x \otimes f+ s_y^*\,{\rm d}y\otimes f + s_z^*\,{\rm d}z\otimes f .
\end{displaymath}
When $s^*\otimes f$ operates on a $p$-vector (or on a $p$-tuple of vectors) $w$, $f(w)$ becomes a real number, and the tensor products with $f(w)$ reduce to multiplying with a real number, and consequently, the value of the map becomes a co-vector
\begin{displaymath}
s^*{\otimes f(w)} = f(w)\,s_x^*{\rm d}x  + f(w)\,s_y^*{\rm d}y + f(w)\,s_z^*{\rm d}z = f(w)s^*
\end{displaymath}
justifying the name "co-vector valued $p$-form".

Finally, for ${\rm End}(E)$-valued forms we first interpret ${\rm End}(E)$ as $E\otimes E^*$, and then $s\otimes s^*$ can be understood as a section of ${\rm End}(E)$. In local coordinates, $(s\otimes s^*)\otimes f$ is written
\begin{displaymath}
(s\otimes s^*)\otimes f =
\begin{bmatrix}
s^xs^*_x \, e_x\otimes {\rm d}x & s^x s^*_y\, e_x\otimes{\rm d}y &  s^xs^*_z\, e_x\otimes {\rm d}z\\
s^ys^*_x\, e_y\otimes {\rm d}x & s^ys^*_y\, e_y\otimes {\rm d}y &  s^ys^*_z\, e_y\otimes {\rm d}z\\
s^zs^*_x\, e_z\otimes {\rm d}x & s^zs^*_y\, e_z\otimes {\rm d}y &  s^zs^*_z\, e_z\otimes {\rm d}z
\end{bmatrix}
\otimes f,
\end{displaymath}
and when $f$ operates on on a $p$-vector or on a $p$-tuple of vectors $w$, the value is
\begin{displaymath}
(s\otimes s^*)\otimes f(w) = f(w)\,(s\otimes s^*),
\end{displaymath}
and hence the name "${\rm End}(E)$-valued $p$-form".

Now, the general field is defined by
\begin{definition}
$F$ is a general field, if it is a section of
\begin{displaymath}
\bigoplus_{p=0}^{n}\bigwedge\nolimits^p T^\ast\Omega,\quad\quad 
\bigoplus_{p=0}^{n}\left(E^p\otimes\bigwedge\nolimits^p T^\ast\Omega\right),\quad\text{or}\quad
\bigoplus_{p=0}^{n}\left({\rm End}(E)\otimes\bigwedge\nolimits^p T^\ast\Omega\right),
 \end{displaymath}
 where $E^p$ is either $E$ or $E^*$.
\end{definition}
\begin{remark}
The general field of ordinary differential forms is of the type $F=\alpha_0 f^0+\alpha_1 f^1+\dots+\alpha_n f^n$, in case of $E$-valued forms $F=\alpha_0 s_0\otimes f^0 + \alpha_1 s_1\otimes f^1+\dots , \alpha_n s_n\otimes f^n$ and so on, where $\alpha_p\in\mathbb{R}$, $p=0,\dots ,n$, and the $s_p$'s are sections of $E$. Consequently, it is plain that any ordinary, $E$-valued, or ${\rm End}(E)$-valued differential form is an instance of the general field.
\end{remark}
  
This construction guarantees that any field from physics field theories expressible as a differential form can also be given as an instance of $F$. Consequently, by operating with general fields we can cover in one token a large class of fields from physics field theories without focusing on any particular one.

\section{Conservation law for general fields}
\label{sec:Conservation law}
In this section we will next introduce the general conservation law as the first order differential equation the general field $F$ should fulfil. For this, we will need some tools.

At first we denote the exterior derivative and the covariant exterior derivative by ${\rm d}$ and ${\rm d}_\nabla$, respectively, where $\nabla$ is a connection \cite{Bleecker} \cite{Crampin-Pirani} \cite{Frankel}  \cite{Baez-Muniain}. 

\subsection{Hodge operator and wedge product} The Hodge operator \cite{Hodge1941} \cite{Flanders} \cite{Bleecker} \cite{Frankel} mapping ordinary $p$-forms to $(n-p)$-forms is denoted by $\star$. The $\star$ depends on the metric tensor. 

Next, the $\star$ and the $\wedge$ need to be extended to $E$-valued and ${\rm End}(E)$-valued forms. For this, we assume a metric on $E$. Then, let $\sigma$ be an ordinary $p$-form and $f=s\otimes \sigma$ an $E$-valued $p$-form.  Furthermore, we denote the set of sections of bundle $E$ by $\Gamma(E)$, and $\sharp:\Gamma(E^*)\rightarrow\Gamma(E)$ is the sharp map, and $\flat:\Gamma(E)\rightarrow\Gamma(E^*)$ is the flat map. 

For our needs the Hodge operator should be defined such that the integral of product $f\wedge {\star f}$ over $\Omega$ becomes a real valued action. For this reason we adopt the following definitions:
\begin{definition}
The $\star$ is an extension of the Hodge operator, if it operates on $E$-valued and ${\rm End}(E)$-valued forms as follows
\begin{alignat*}{1}
\star:\,\, &\Gamma(E\otimes \bigwedge\nolimits^p T^\ast\Omega)\,\, \rightarrow\,\, \Gamma (E^*\otimes \bigwedge\nolimits^{n-p} T^\ast\Omega), \,s\otimes \sigma\mapsto \flat s\otimes{\star{\sigma}}, \\
\star:\,\,&\Gamma(E^*\otimes \bigwedge\nolimits^p T^\ast\Omega)\,\, \rightarrow\,\, \Gamma (E\otimes \bigwedge\nolimits^{n-p} T^\ast\Omega), \, s\otimes \sigma\mapsto \sharp s\otimes{\star{\sigma}}, \\
\star:\,\,&\Gamma({\rm End}(E)\otimes \bigwedge\nolimits^p T^\ast\Omega) \rightarrow \Gamma ({\rm End}(E)\otimes \bigwedge\nolimits^{n-p} T^\ast\Omega), \, S\otimes\sigma\mapsto S\otimes{\star{\sigma}}.
\end{alignat*}
\end{definition}

\begin{definition}
The $\wedge$ is an extension of the wedge product between ordinary differential forms, if it operates on $E$-valued and ${\rm End}(E)$-valued forms as follows
\begin{alignat*}{1}
\wedge: \, \,& \Gamma(E\otimes \bigwedge\nolimits^p T^\ast\Omega)\,\,\times\,\,\Gamma(E^*\otimes \bigwedge\nolimits^r T^\ast\Omega)\,\,\rightarrow\,\,\Gamma((E\otimes E^*)\otimes \bigwedge\nolimits^{p+r} T^\ast\Omega),\\
& (s\otimes \sigma,s'\otimes \sigma')\mapsto (s\otimes s')\otimes (\sigma\wedge \sigma')
\end{alignat*}
and
\begin{alignat*}{1}
\wedge: \,\,& \Gamma({\rm End}(E)\otimes \bigwedge\nolimits^p T^\ast\Omega)\,\,\times\,\,\Gamma({\rm End}(E)\otimes \bigwedge\nolimits^r T^\ast\Omega)\,\,\rightarrow\,\,\Gamma({\rm End}(E)\otimes \bigwedge\nolimits^{p+r} T^\ast\Omega),\\
& (S\otimes \sigma, S'\otimes \sigma')\mapsto SS'\otimes (\sigma\wedge \sigma').
\end{alignat*}
\end{definition}

\begin{remark} 
The wedge product of  $s\otimes\sigma$ and ${\star(s\otimes\sigma)}$ yields an $E\otimes E^*$-valued form $(s\otimes \flat s)\otimes (\sigma\wedge\star \sigma)$.
\end{remark}
\begin{remark} 
The wedge product of $S\otimes \sigma$ and $\star(S\otimes \sigma)$ is an ${\rm End}(E)$-valued form
\begin{displaymath}
S\otimes \sigma\wedge\star(S\otimes \sigma) = SS\otimes(\sigma\wedge{\star\sigma})\, .
\end{displaymath}
\end{remark}

\subsection{Action} Hereinafter we need not to emphasize explicitly whether we talk of ordinary, $E$-valued, or ${\rm End}(E)$-valued $p$-forms. Consequently, symbols $f^p$, $g^p$, $h^p$, etc. may denote to $p$-forms of any type.

Physics field theories are typically built around the idea of conserving some fundamental notion, such as energy or probability, and the very idea is to express such notion as a $L^2$-norm of a product of fields integrated over the whole manifold. To introduce such type of real valued action, we define first:
\begin{definition}
Given a vector space $V$, a linear map $V\otimes V^*\rightarrow\mathbb{R}$ is the trace ${\rm tr}$, if $v\otimes \nu\mapsto \nu(v)$. 
\end{definition}
\begin{remark}
The trace is independent of the choice of basis.
\end{remark}

With the aid of the trace we are able to introduce function $\mathcal{A}$ mapping $f^p$'s to real numbers:
\begin{displaymath}
{\mathcal A}(f^p)\, = \, \frac{1}{2}\int\limits_\Omega {\rm tr}(f^p\wedge{\star f^p}).
\end{displaymath}
This provides us with an action of the desired type; the integral of the product between the elements of pair $\{f^p, g^{n-p}\}$ over $\Omega$, where the elements $f^p$ and $g^{n-p}$ are related by a constitutive relation $g^{n-p}={\star f^p}$, yields a real number. 

Physics field theories call also for source terms. To append them to action ${\mathcal A}$, we need to add another term. Let the sources be given as a general field $G$:
\begin{displaymath}
G=\sum\limits_{p=0}^n g^p.
\end{displaymath}
Complementing the source terms to action ${\mathcal A}$ results in
\begin{displaymath}
{\mathcal A}^p\, =\, \frac{1}{2}\int\limits_\Omega {\rm tr}(f^p\wedge{\star f^p}) + \int\limits_\Omega {\rm tr}(a^{p-1}\wedge {\star g^{p+1}}),
\end{displaymath}
where $a^p$ is the potential. Notice, however, the potential does not have the same interpretation in classical field theories and in gauge theories. For example, in classical theories one typically writes $f^p={\rm d}_\nabla a^{p-1}$, whereas in case of Yang-Mills theory \cite{Yang-Mills1954} \cite{Yang2014} \cite{Baez-Muniain} \cite{Tao2004} the ${\rm End}(E)$-valued curvature $2$-form $f^2$ is $f^2=f_0 +{\rm d}_\nabla a^1+ a^1\wedge a^1$, where $a^1$ is the vector potential and $f_0$ is the curvature of the connection, see \cite[p. 274-5]{Baez-Muniain}, \cite{Tao2004}. 
\begin{remark}
Notice, the wedge product between an ${\rm End}(E)$-valued with itself need not to vanish. This is due the non-commutativity between the two products involved. For more detail, see \cite[p.280-281]{Crampin-Pirani}.
\end{remark}

Insisting the variation $\delta{\mathcal A}^p$ to vanish yields differential equations. This is the so called action principle \cite{Feynman} \cite{Baez-Muniain} \cite{Sanders2014}. In case of  ${\mathcal A}^p$ potential $a^p$ is varied by $a^p_\alpha = a^p + \alpha \delta a^p$, $\alpha \in\mathbb{R}$, and then the variation $\delta{\mathcal A}^p$ is given by 
\begin{displaymath}
\delta {\mathcal A}^p \, = \,\frac{\rm d}{{\rm d}\alpha } {\mathcal A}^p(a^{p-1}_\alpha )\bigg\rvert_{\alpha =0}\, .
\end{displaymath}
Thereafter, the differential equations are obtained with an integration by parts process. For details and examples see \cite{Baez-Muniain} \cite{Bleecker} \cite{Frankel}. All the examples we will show later on in section \ref{sec:Physics models} rely on the action ${\mathcal A}^p$ given above.

\subsection{General conservation law in space-time} We have now all what is needed to define the class of field theoretical problems we are interested in: (if the connection is trivial, ${\rm d}_\nabla$ reduces to ${\rm d}$.)
\begin{definition}
On manifold $\Omega$ differential equations
\begin{displaymath}
{\rm d}_\nabla F \, =\, 0\,  \quad\text{and}\quad {\star{\rm d}_\nabla{\star F}}\, = \, {\star{\star G}}
\end{displaymath}
form the general conservation law for pair $\{F,G\}$, if the equations are derivatives from a real valued action.
\end{definition}
\begin{remark}
The name conservation law follows from the idea that small changes in the solutions of the differential equations do not change the underlying action up to the first order. In this sense the differential equations have to do with conservation of energy, probability or of some other significant notion. 
\end{remark}
\begin{remark}
Operator ${\star{\star}}$ specifies the sign, ${\star{\star}} = (-1)^{p(n-p)+1}$. 
\end{remark}

\begin{lemma}
\label{lm: Conservation law in 4d}
In dimension $n=4$ with the given signature the Hodge operator $\star$ maps $p$-forms to $(4-p)$-forms, and it fulfils $\star^2=(-1)^{p(4-p)+1} $. Consequently, the general conservation law for pair $\{F,G\}$, $F=\sum_{p=0}^4 f^p$, and $G=\sum_{p=0}^4 g^p$, can be equivalently given by
\begin{displaymath}
\setlength{\arraycolsep}{1pt}
\begin{bmatrix}
& -{\star{\rm d}_\nabla{\star}} & & \\
{\rm d}_\nabla & &{\star{\rm d}_\nabla{\star}} &  \\
& {\rm d}_\nabla& & -{\star{\rm d}_\nabla{\star}} \\
& & {\rm d}_\nabla  & & {\star{\rm d}_\nabla{\star}} \\
& & & {\rm d}_\nabla  & \\
\end{bmatrix}
\begin{bmatrix}
f^0 \\
f^1 \\
f^2 \\
f^3 \\
f^4 \\
\end{bmatrix}
 \, = \,
\begin{bmatrix}
g^0 \\
g^1 \\
g^2 \\
g^3 \\
g^4 \\
\end{bmatrix} \, .
\vspace*{3.0mm}\\
\end{displaymath}
\end{lemma}

\subsection{General conservation law in space and time} Let us next assume a decomposition of $\Omega$ into space and time, $\Omega=\Omega_s\times\Omega_t$. The space-like part of the exterior covariant derivative ${\rm d}_\nabla$ is denoted by ${\rm d}^s_\nabla$. 

Here we state our main result as theorem \ref{thr:Conservation law in space and time} and corollary \ref{crl:Conservation law in space and time II} following from lemma \ref{lm: Conservation law in 4d}:
\begin{theorem}
\label{thr:Conservation law in space and time}
In $\Omega_t\times\Omega_s$ and in dimension $n=1+3$ the conservation law is given by
\begin{displaymath}
\begin{footnotesize}
\renewcommand{\arraystretch}{1.65}
\setlength{\arraycolsep}{1.5pt}
\begin{bmatrix}
{{\rm d}t}\wedge\partial_t &  & & & &  {\rm d}_{\nabla}^s& & \\ 
&{\star{({\rm d}t\wedge{\partial_t})}\star}& & &{\rm d}_{\nabla}^s &  & & \\ 
& &{{\rm d}t}\wedge\partial_t & &  & -{\star{{\rm d}_\nabla^s\star}}& &{\rm d}_{\nabla}^s     \\ 
&  &  &{\star{({\rm d}t\wedge{\partial_t})}\star}&  {\star{{\rm d}_\nabla^s\star}} & & {\rm d}_{\nabla}^s  & \\ 
&  {\star{{\rm d}_\nabla^s\star}} & & {\rm d}_{\nabla}^s & {{\rm d}t}\wedge\partial_t& & & \\ 
-{\star{{\rm d}_\nabla^s\star}} & & {\rm d}_{\nabla}^s & & &-{\star{({\rm d}t\wedge{\partial_t})}\star}& & \\ 
& &  &{\star{{\rm d}_\nabla^s\star}}& & & {{\rm d}t}\wedge\partial_t &\\ 
& & -{\star{{\rm d}_\nabla^s\star}}& & & & & -{\star{({\rm d}t\wedge{\partial_t})}\star}
\end{bmatrix}
\begin{bmatrix}
      f^{3}_s \\ 
      f^{4} _t\\ 
      f^{1}_s \\ 
      f^{2}_t \\ 
      f^{2} _s\\ 
      f^{3}_t \\ 
     f^0\\ 
     f^1_t
\end{bmatrix}
\, = \, 
\begin{bmatrix}
    g^{4}_t \\ 
    g^{3}_s \\ 
    g^{2}_t \\ 
    g^{1}_s \\ 
    g^{3}_t \\ 
    g^{2}_s  \\ 
    g^1_t \\ 
    g^0
\end{bmatrix}
\end{footnotesize}
\end{displaymath}
\end{theorem}

\begin{proof}
The theorem follows from the following two properties: 
\begin{enumerate}
\item The covariant exterior derivative can be decomposed into a time and space-like part by writing
\begin{displaymath}
{\rm d}_\nabla  = {\rm d}t\wedge\partial_t  + {\rm d}^s_\nabla \, .
\end{displaymath}  
\item For any $p>0$, any $p$-form $f^p$ can be decomposed into a sum of a time-like and space-like component:
\begin{displaymath}
f^p = f^p_t + f^p_s,
\end{displaymath}
where $f^p_s$ involves only space-like components of $f^p$, and $f^p_t$ contains the remaining components with a time-like part.
\end{enumerate}
These two properties imply
\begin{displaymath}
{\rm d}_\nabla f^p = {\rm d}t\wedge\partial_t  (f^p_t + f^p_s) + {\rm d}^s_\nabla (f^p_t + f^p_s)\,\quad \forall p>0\, .
\end{displaymath}
Since the wedge product between ${\rm d}t$ and any $f^p_t$ vanishes, one may equivalently write
\begin{displaymath}
{\rm d}_\nabla f^p = {\rm d}t\wedge\partial_t  f^p_s + {\rm d}^s_\nabla f^p_t + {\rm d}^s_\nabla f^p_s,\quad \forall p>0\, .
\end{displaymath}
In case of $p=0$ we have
\begin{displaymath}
{\rm d}_\nabla f^0 = {\rm d}t\wedge\partial_t  f^0 + {\rm d}^s_\nabla f^0\, .
\end{displaymath}

Next, a decomposition of ${\star{{\rm d}_\nabla \star}F}$ into a time and space-like parts yields 
\begin{displaymath}
{\star{{\rm d}_\nabla \star}F} \, = \, {\star({\rm d}t\wedge\partial_t  + {\rm d}^s_\nabla){\star}}F \, = \, 
{\star{{\rm d}t\wedge\partial_t\star}}F + \star \, {{\rm d}^s_\nabla{\star}}F.
\end{displaymath}
Consequently, in case of $p=n$ we have
\begin{displaymath}
{\star{{\rm d}_\nabla\star}} f^n =  {\star{\rm d}t}\wedge\partial_t  {\star f^n_t} + {\star{{\rm d}^s_\nabla\star}} f^n_t \, .
\end{displaymath}

For $p<n$ notice first, that any $(n-p)$-form $\star f^p_t$ is of the type $f^{n-p}_s$, and $\star f^p_s$ is of the type $f^{n-p}_t$. Accordingly,
the wedge product between ${\rm d}t$ and any ${\star f^p_s}$ has to vanish, and the derivative is given by
\begin{alignat*}{1}
{\star{{\rm d}_\nabla\star}} f^p & \, = \, {\star{\rm d}t}\wedge\partial_t  {\star(f^p_t + f^p_s)} \, + \, {\star{{\rm d}^s_\nabla\star}} (f^p_t + f^p_s)\\
                                                & \, = \, {\star{\rm d}t}\wedge\partial_t  {\star f^p_t} + {\star{{\rm d}^s_\nabla\star}} f^p_t \, + \,  {\star{{\rm d}^s_\nabla\star}}  f^p_s\,  .
\end{alignat*}

The proof follows now from: $n=4$, $F=\sum_{p=0}^n f^p$,
\begin{alignat*}{1}
{\rm d}_\nabla f^0 & \, = \,  {\rm d}t\wedge\partial_t  f^0 + {\rm d}^s_\nabla f^0\, , \\
{\rm d}_\nabla f^p & \, = \,  {\rm d}t\wedge\partial_t  f^p_s + {\rm d}^s_\nabla f^p_t + {\rm d}^s_\nabla f^p_s,\quad \forall p>0\,, \\
{\star{{\rm d}_\nabla\star}} f^n &\, = \,   {\star{\rm d}t}\wedge\partial_t  {\star f^n_t} + {\star{{\rm d}^s_\nabla\star}} f^n_t \, , \\
{\star{{\rm d}_\nabla\star}} f^p & \, = \,  {\star{\rm d}t}\wedge\partial_t  {\star f^p_t} + {\star{{\rm d}^s_\nabla\star}} f^p_t \, + \,  {\star{{\rm d}^s_\nabla\star}}  f^p_s\, ,\quad \forall p<n\, ,
\end{alignat*}
and furthermore, from ${\rm d}^s_\nabla f^3_s\equiv 0$ and ${\star{{\rm d}^s_\nabla\star}} f^1_t\equiv 0$; Substituting these back to lemma \ref{lm: Conservation law in 4d} yields
\begin{displaymath}
\begin{footnotesize}
\renewcommand{\arraystretch}{1.65}
\setlength{\arraycolsep}{1.5pt}
\begin{bmatrix}
 & -{\star{({\rm d}t\wedge{\partial_t})}\star}& -{\star{{\rm d}_\nabla^s\star}}& & & & &\\
 {{\rm d}t}\wedge\partial_t& & & {\star{{\rm d}_\nabla^s\star}} & & & &\\
 {\rm d}_{\nabla}^s & & &{\star{({\rm d}t\wedge{\partial_t})}\star} & {\star{{\rm d}_\nabla^s\star}} & & &\\
 & {\rm d}_{\nabla}^s & {{\rm d}t}\wedge\partial_t & & & -{\star{{\rm d}_\nabla^s\star}} & \\
 & & {\rm d}_{\nabla}^s & & &-{\star{({\rm d}t\wedge{\partial_t})}\star}& -{\star{{\rm d}_\nabla^s\star}}& \\
 & & & {\rm d}_{\nabla}^s & {{\rm d}t}\wedge\partial_t & & & {\star{{\rm d}_\nabla^s\star}} & \\
 & & & & {\rm d}_{\nabla}^s & & &{\star{({\rm d}t\wedge{\partial_t})}\star}\\
 & & & & & {\rm d}_{\nabla}^s & {{\rm d}t}\wedge\partial_t 
\end{bmatrix}
\begin{bmatrix}
      f^0 \\ 
      f^1 _t\\ 
      f^1 _s\\ 
      f^2 _t\\ 
      f^2_s\\ 
      f^3_t \\ 
      f^3_s\\ 
      f^4 _t
\end{bmatrix}
=
\begin{bmatrix}
      g^0 \\ 
      g^1 _t\\ 
      g^1 _s\\ 
      g^2 _t\\ 
      g^2_s\\ 
      g^3_t \\ 
      g^3_s \\ 
      g^4 _t
\end{bmatrix}
,
\end{footnotesize}
\end{displaymath}
and the desired result is obtained by row and column swapping.
\end{proof}

Let $f^p_s$, $g^p_s$, $F^p_s$, and $G^p_s$ denote $p$-forms on $\Omega_s$, and $\star_s$ is the Hodge operator in the space-like component $\Omega_s$.

\begin{corollary}
\label{crl:Conservation law in space and time II}
Theorem \ref{thr:Conservation law in space and time} is equivalent to
\begin{displaymath}
\begin{footnotesize}
\renewcommand{\arraystretch}{1.65}
\setlength{\arraycolsep}{1pt}
\begin{bmatrix}
\partial_t &  & & & &  -{\rm d}_{\nabla}^s& & \\ 
&{\star_s{\partial_t\star_s}}& & &{\rm d}_{\nabla}^s &  & & \\ 
& &\partial_t & &  & {\star_s{{\rm d}_\nabla^s\star_s}}& &-{\rm d}_{\nabla}^s     \\ 
&  &  &{\star_s{\partial_t\star_s}}&  {\star_s{{\rm d}_\nabla^s\star_s}} & & {\rm d}_{\nabla}^s  & \\ 
&  {\star_s{{\rm d}_\nabla^s\star_s}} & & -{\rm d}_{\nabla}^s & \partial_t& & & \\ 
{\star_s{{\rm d}_\nabla^s\star_s}} & & {\rm d}_{\nabla}^s & & &-{\star_s{\partial_t\star_s}}& & \\ 
& &  &{\star_s{{\rm d}_\nabla^s\star_s}}& & & \partial_t &\\ 
& & {\star_s{{\rm d}_\nabla^s\star_s}}& & & & & -{\star_s{\partial_t\star_s}}
\end{bmatrix}
\begin{bmatrix}
      f^{3}_s\\ 
      F^{3}_s\\ 
      f^{1}_s\\ 
      F^{1}_s\\ 
      f^{2}_s\\ 
      F^{2}_s\\ 
      f^0\\ 
      F^0
\end{bmatrix}
\, = \, 
\begin{bmatrix}
    G^{3}_s\\ 
    g^{3}_s \\ 
    G^{1}_s \\ 
    g^{1}_s \\ 
    G^{2}_s \\ 
    g^{2}_s \\ 
    G^0 \\ 
    g^0
\end{bmatrix}
\end{footnotesize}
\end{displaymath}
\end{corollary}
\begin{proof}
Let $p$-form $f^p$ be decomposed into
\begin{displaymath}
f^p = f^p_t + f^p_s,
\end{displaymath}
as in the proof of theorem \ref{thr:Conservation law in space and time}.  

The exterior derivative satifies
\begin{displaymath}
{\rm d}(f^p\wedge g^r) = {\rm d}f^p\wedge g^r + (-1)^p f^p\wedge  {\rm d}g^r.
\end{displaymath}
Hence, the space-like exterior derivative of ${\rm d}^s_\nabla f^p$ can be given by
\begin{displaymath}
 {\rm d}^s_\nabla f^p = {\rm d}^s_\nabla f^p_t + {\rm d}^s_\nabla f^p_s = {\rm d}^s_\nabla({\rm d}t\wedge F^{p-1}_s) + {\rm d}^s_\nabla f^p_s = -{\rm d}t\wedge{\rm d}^s_\nabla F^{p-1}_s +  {\rm d}^s_\nabla f^p_s,
\end{displaymath}
where $F^{p-1}_s$ is a space-like $(p-1)$-form.

The $\star$ operator in $\Omega$ can be given in terms of $\star_s$ as follows
\begin{displaymath}
\star f^p = \star(f^p_t + f^p_s) = \star({\rm d}t\wedge F^{p-1}_s + f^p_s) = -{\star_s F^{p-1}_s} + (-1)^p{\rm d}t\wedge {\star_s f^p_s}\, .
\end{displaymath}
Consequently, we also have
\begin{displaymath}
{\rm d}^s_\nabla{\star f^p} = -{\rm d}^s_\nabla{\star_s F^{p-1}_s} + (-1)^{p+1}{\rm d}t\wedge {\rm d}^s_\nabla{\star_s f^p_s}\, ,
\end{displaymath}
and recursively
\begin{displaymath}
{\star{\rm d}^s_\nabla{\star f^p}} = (-1)^p{\rm d}t\wedge{\star_s{\rm d}^s_\nabla{\star_s F^{p-1}_s}} + (-1)^p{\star_s{\rm d}^s_\nabla{\star_s f^p_s}}\, .
\end{displaymath}

The time derivatives satifies
 \begin{displaymath}
{\rm d}t\wedge\partial_t f^p = {\rm d}t\wedge\partial_t ({\rm d}t\wedge F^{p-1} + f^p_s) = {\rm d}t\wedge\partial_t f^p_s,
\end{displaymath}
and
\begin{displaymath}
{\star{({\rm d}t\wedge{\partial_t})}\star} f^p = -{\star{({\rm d}t\wedge{\partial_t})}{\star_s F^{p-1}_s}} = {\star_s\partial_t{\star_s F^{p-1}_s}}.
\end{displaymath}

Summing up, we have the following
\begin{alignat*}{1}
{\rm d}^s_\nabla f^p_t & = -{\rm d}t\wedge{\rm d}^s_\nabla F^{p-1}_s, \\
{\star{\rm d}^s_\nabla{\star f^p_s}} &=(-1)^p{\star_s{\rm d}^s_\nabla{\star_s f^p_s}}, \\
{\star{\rm d}^s_\nabla{\star f^p_t}} & = (-1)^p{\rm d}t\wedge{\star_s{\rm d}^s_\nabla{\star_s F^{p-1}_s}}, \\
{\rm d}t\wedge\partial_t f^p &= {\rm d}t\wedge\partial_t f^p_s,\\
{\star{({\rm d}t\wedge{\partial_t})}\star} f^p_t & = {\star_s\partial_t{\star_s F^{p-1}_s}}.
\end{alignat*}
Substituting this back to theorem \ref{thr:Conservation law in space and time} yields
\begin{displaymath}
\begin{footnotesize}
\renewcommand{\arraystretch}{1.85}
\setlength{\arraycolsep}{0.8pt}
\begin{bmatrix}
{{\rm d}t}{\wedge\partial_t} &  & & & & -{\rm d}t{\wedge{\rm d}_{\nabla}^s}& & \\ 
&{\star_s{\partial_t\star_s}}& & &{\rm d}_{\nabla}^s &  & & \\ 
& &{{\rm d}t}{\wedge\partial_t}& &  & {\rm d}t{\wedge{\star_s{{\rm d}_\nabla^s\star_s}}}& &-{\rm d}t{\wedge{\rm d}_{\nabla}^s}  \\ 
&  &  &{\star_s{\partial_t\star_s}}&  {\star_s{{\rm d}_\nabla^s\star_s}} & & {\rm d}_{\nabla}^s  & \\ 
&{\rm d}t{\wedge{\star_s{{\rm d}_\nabla^s\star_s}}} & & -{\rm d}t{\wedge{\rm d}_{\nabla}^s} & {{\rm d}t}{\wedge\partial_t}& & & \\ 
{\star_s{{\rm d}_\nabla^s\star_s}} & & {\rm d}_{\nabla}^s & & &-{\star_s{\partial_t\star_s}}& & \\ 
& &  &{{\rm d}t}{\wedge{\star_s{{\rm d}_\nabla^s\star_s}}}& & & {{\rm d}t}{\wedge\partial_t}&\\ 
& & {\star_s{{\rm d}_\nabla^s\star_s}}& & & & & -{\star_s{\partial_t\star_s}}
\end{bmatrix}
\begin{bmatrix}
      f^{3}_s\\ 
      F^{3}_s\\ 
      f^{1}_s\\ 
      F^{1}_s\\ 
      f^{2}_s\\ 
      F^{2}_s\\ 
      f^0\\ 
      F^0
\end{bmatrix}
\, = \, 
\begin{bmatrix}
   {\rm d}t{\wedge G^{3}_s}\\ 
    g^{3}_s \\ 
    {\rm d}t{\wedge G^{1}_s}\\ 
    g^{1}_s \\ 
    {\rm d}t{\wedge G^{2}_s}\\ 
    g^{2}_s  \\ 
    {\rm d}t{\wedge G^0} \\ 
    g^0
\end{bmatrix}
\end{footnotesize}
\end{displaymath}
All the entries of the odd rows have a wedge product with ${{\rm d}t}$ from the left, and the proof follows by taking this wedge product as a common factor.
\end{proof}

\begin{remark}
\label{rmk:metric proxy}
Let $\Omega_s$ be a 3-dimensional Euclidean space. In terms of classical vector analysis the general conservation law of corollary \ref{crl:Conservation law in space and time II} corresponds with
\begin{displaymath}
\begin{footnotesize}
\renewcommand{\arraystretch}{1.75}
\begin{bmatrix}
{\partial_t}&                   &              &                  &                 &  -{\rm div}  &                &                   \\ 
               & {\partial_t}\alpha_\omega&               &                  &  {\rm div} &                  &                 &                   \\ 
               &                  &{\partial_t}&                  &                 &  {\rm curl}\,\alpha_v &                & -{\rm grad} \\ 
               &                  &               &  {\partial_t}\alpha_s & {\rm curl}\,\alpha_u  &                  & {\rm grad}&                \\ 
               & {\rm  grad}\,\alpha_\omega&                & -{\rm curl}  & {\partial_t} &                  &              &                    \\ 
{\rm  grad}\,\alpha_\phi&              &  {\rm curl} &                 &                 &{-\partial_t}\,\alpha_v&                 &                   \\ 
               &                 &                  & {\rm div}\,\alpha_s  &                &                  &{\partial_t} &                    \\ 
               &                 &  {\rm div}\,\alpha_r  &                 &                &                  &                 & {-\partial_t}\alpha_\psi 
\end{bmatrix}
\begin{bmatrix}
      \phi \\ 
      \omega \\ 
      \overline{r} \\ 
      \overline{s} \\ 
      \overline{u} \\ 
      \overline{v} \\ 
      \varphi \\ 
      \psi
\end{bmatrix}
\, = \,
\begin{bmatrix}
    \theta \\ 
    \lambda \\ 
    \overline{a} \\ 
    \overline{c} \\ 
    \overline{l} \\ 
    \overline{m}  \\ 
    \kappa \\ 
    \beta
\end{bmatrix}\, ,
\end{footnotesize}
\end{displaymath}
where $\alpha_\phi$, $\alpha_\phi$, $\dots$, $\alpha_\psi$ are parameters that follow from the possible change of type when working
out the (1-) vectors corresponding with ${\star_s f^p_s}$ and ${\star_s F^p_s}$ in Euclidean space. Functions $\beta$, $\kappa$, $\lambda$, $\theta$, $\varphi$, $\psi$, $\phi$, and $\omega$ are scalar fields, and $\overline{a}$, $\overline{c}$, $\overline{l}$, $\overline{m}$, $\overline{r}$, $\overline{s}$, $\overline{u}$, and $\overline{v}$ are vector fields. (In case of $E$-valued forms, the corresponding entries of the column vectors $[\phi, \omega, \dots , \psi]^T$ and $[\theta,\lambda, \dots , \beta]^T$ should be considered as vector-valued.)
\end{remark}
\begin{proof}
In an $n$-dimensional Euclidean space  any $p$-form $f^p$ can be identified with a $p$-vector $f_p$; For all $p$-vectors $u_p$
\begin{displaymath}
(f_p, u_p) = f^p(u_p)
\end{displaymath}
should hold. Now, in particular, in dimension three operators ${\rm grad}$, ${\rm curl}$, and ${\rm div}$ are defined such that for all
$u_1$, $u_2$, and $u_3$
\begin{alignat*}{2}
({\rm grad} f_0, u_1) &= {\rm d}f^0(u_1),\\
({\rm curl}f_1, {\star_s u_2})   &= {\rm d}f^1(u_2),\\
({\rm div}({\star_s f_2}), {\star_s u_3})   &= {\rm d}f^2(u_3).
\end{alignat*}
hold, and the proof follows from this identification.
\end{proof}

\section{Physics field models as instances of the conservation law}
\label{sec:Physics models}

A large class of particular field models from physics can now be given as instances of the conservation law given in corollary \ref{crl:Conservation law in space and time II}, as demonstrated in the next subsections. 

\subsection{Electromagnetism}

We start with electromagnetism. In $\Omega=\Omega_t\times\Omega_s$, the Faraday field becomes $b + e\wedge {\rm d}t$, where $b$ represents magnetic flux density and $e$ represents the electric field strength. Both $b$ and $e \wedge {\rm d}t$ are $2$-forms. The source term of electromagnetism is $\star(-j\wedge {\rm d}t + q) $, where $2$-form $j$ is current density and $3$-form $q$ is charge density. 

To substitute $F=f^2_s+f^2_t=b+ e \wedge {\rm d}t= b + {\rm d}t\wedge(-e)$ and $G=g^1_s+g^1_t=\star(j\wedge {\rm d}t + q)= {\star_s j} -{\rm d}t\wedge{\star_s q}$ to corollary \ref{crl:Conservation law in space and time II}, we set $F^1_s = -e$, $f^2_s=b$, $g^1_s={\star_s j}$, and $G^0={-{\star_s q}}$, and then, the general conservation law yields
\begin{alignat*}{3}
{\rm d}^s\, b\, & =\, 0\,, & \text{2. row} &, \\
{\rm d}^s e \, +\, \partial_t b\, & =\, 0, & \text{5. row} &,\\
-{\star_s{\partial_t \star_s}}\, e\, + \, \star_s {{\rm d}^s \star_s}\, b\, & =\, {{\star_s j}}, & \text{4. row} &,\\
-{\star_s {{\rm d}^s \star_s}}\, e\, & =\, {-{\star_s q}}, & \quad \text{7. row}& .
\end{alignat*}
This is an equivalent expression for Maxwell's equations
\begin{alignat*}{3}
{\rm d}^s\, b\, &=\, 0  & {\rm d}^s e \, +\, \partial_t b \, &= \, 0, \\
 {\rm d}^s h \, &= \, j + \partial_t d \quad\quad\quad\quad & {\rm d}^s d \, &=\,  q,
\end{alignat*}
together with the constitutive laws $h=\star_\nu b$ and $d=\star_\epsilon e$, where $\star_\nu = \nu\star_s$ and $\star_\epsilon=\epsilon\star_s$. (The material parameters can also be embedded into the Hodge operator, see \cite{Bossavit2001}.)

To give the same in terms of classical vector analysis, one first selects $\overline{u}=b$, $\overline{s}=-e$, $\overline{l}=j$,$\kappa=-q$, $\alpha_s=\epsilon$,  and $\alpha_u = \nu$. Then, the general conservation law in remark \ref{rmk:metric proxy} yields
\begin{alignat*}{3}
{\rm div}\, b\, & =\, 0\,, & \text{2. row} &, \\
{\rm curl}\, e \, +\, \partial_t b & =\, 0, & \text{5. row} &,\\
-{\partial_t}\,\epsilon e \, + \, {\rm curl}\,\nu b\, & =\,  j, & \text{4. row} &,\\
-{\rm div}\, \epsilon e \, & =\, -q & \quad \text{7. row}& .
\end{alignat*}

\subsection{Schr\"odinger equation}

The non-relativistic Schr\"odinger equation can also be derived from the general conservation law as a coupled pair of differential equations. The components of the wave function by  are denoted by $\varphi_R$ and $\varphi_I$. In addition, we introduce a pair $\{q_R, q_I\}$ of auxiliary variables. Then we write 
\begin{alignat*}{3}
F^0 & =  \hslash\,\varphi_R,\quad\quad\quad &f^3_s & = \hslash\,\varphi_I,\\
f^1_s & = \frac{\hslash}{2m} q_R, &F^2_s & = \frac{\hslash}{2m} q_I, \\
G^1_s & = q_R,  &g^2_s & = q_I, \\
g^0 &= -V\varphi_r, &G^3_s & = -V\varphi_I,
\end{alignat*}
and
\begin{displaymath}
-{\star_s\varphi_R}\, = \, \varphi_I. 
\end{displaymath}
Be aware, although the subscripts $R$ and $I$ correspond with what are conventionally expressed as the real and imaginary components, respectively, here, they are here just labels.

Substituting the $R$-labeled terms $F^0$, $f^1_s$, $G^1_s$, and $g^0$ back into corollary \ref{crl:Conservation law in space and time II} results in
\begin{alignat*}{3}
\frac{\hslash}{2m} \partial_t q_R -\hslash\, {\rm d}^s\varphi_R\,  & =\,  q_R\,, \quad &&\text{3. row}, \\
\frac{\hslash}{2m}\, {\rm d}^s q_R \, & =\, 0, &&\text{6. row},\\
\frac{\hslash}{2m}\,{{\star_s{\rm d}^s}\star_s}q_r - \hslash{{\star_s{\partial}_t}\star_s}\varphi_r & =\, -V\varphi_R, \quad &&\text{8. row}.
\end{alignat*}
In the non-relativistic case,  $\partial_t q_R$ is neglected by a modelling decision \cite{Kravchenko_2005} implying the third row can be written as $q_R=-\hslash\, {\rm d}^s\varphi_R$. When this is substituted into the equation from the sixth row, that becomes redundant since ${\rm d}^s{\rm d}^s\equiv 0$. What remains is the eight row, and by substituting $q_R=-\hslash\,{\rm d}^s\varphi_R$ into it results in
\begin{displaymath}
-\hslash{{\star_s{\partial}_t}\star_s}\varphi_R - \frac{\hslash^2}{2m} {{\star_s{\rm d}^s}\star_s}{\rm d}^s\varphi_R \,  =\,  -V\varphi_R\,. 
\end{displaymath}
Now, as we have ${\star_s\varphi_R} = -\varphi_I$, the equation arising from the eight row is equivalent to 
 \begin{displaymath}
\hslash\,{{\star_s{\partial}_t}\varphi_I} - \frac{\hslash^2}{2m} {{{\rm d}^s\star_s}{\rm d}^s\star_s}\varphi_R \,  =\,  -V\varphi_R\,. 
\end{displaymath}
Notice, ${{\star_s{\rm d}^s}\star_s}{\rm d}^s$ is the Laplace operator \cite{Grabowska_2008}.

Symmetrically, substituting the $I$-labeled terms $f^3_s$, $F^2_s$, $g^2_s$, and $G^3_s$ back into corollary \ref{crl:Conservation law in space and time II} yields
\begin{alignat*}{3}
\hslash\,{{\star_s{\rm d}^s}\star_s}\varphi_I - \frac{\hslash}{2m}\,{{\star_s{\partial}_t}\star_s}q_I & =\, q_I,\quad &&\text{6. row}, \\
\frac{\hslash}{2m}\,{{\star_s{\rm d}^s}\star_s} q_I \, & =\, 0, &&\text{3. row},\\
\hslash\, \partial_t\varphi_I\ - \frac{\hslash}{2m} {\rm d}^s q_I \,  & =\,  -V\varphi_I,  \quad &&\text{1. row}.
\end{alignat*}

Since $\partial_t q_R$ is neglected, so is also $-\partial_t q_I=\star_s\partial_t \varphi_R$, and this implies the equation of the third row of the conservation law is a tautology. Then, we have $q_I= \hslash\,{{\star_s{\rm d}^s}\star_s}\varphi_I$, and when this and $-{\star_s\varphi_R}= \varphi_I$ is substituted to the equation of the first row
 one gets
 \begin{displaymath}
-\hslash\,{{\star_s{\partial}_t}\varphi_R} - \frac{\hslash^2}{2m}{{\star_s{\rm d}^s}\star_s}{\rm d}^s\varphi_I \,  =\,  -V\varphi_I\,. 
\end{displaymath}
The Schrödinger equation is now the following pair of equations
\begin{alignat*}{1}
\hslash\,{{\star_s{\partial}_t}\varphi_R} + \frac{\hslash^2}{2m}{{\star_s{\rm d}^s}\star_s}{\rm d}^s\varphi_I \, &=\,  V\varphi_I\, ,\\
\hslash\,{{\star_s{\partial}_t}\varphi_I} - \frac{\hslash^2}{2m} {{{\rm d}^s\star_s}{\rm d}^s\star_s}\varphi_R \,  &=\,  -V\varphi_R\, .
\end{alignat*}
This yields the same solutions as the textbook expression of the Schrödinger equation
\begin{equation}
\hslash\,\partial_t \varphi - {\rm i}\frac{\hslash^2}{2m}\,{\rm div}\,{\rm grad}\, \varphi = -{\rm i}V\varphi,
\end{equation}
where $\varphi$ is complex valued scalar function.

The Dirac equation and the Gross-Pitaevskii equation\cite{Gross1961} \cite{Pitaevskii1961} \cite{Seiringer2002} are concretised in a similar fashion\cite{Rabina2018}. The Klein-Gordon equation is akin to the Schr\"odinger equation but second order in time.

\subsection{Elasticity}

To demonstrate deriving the basic equations of small-strain elasticity \cite{Abraham-Marsden} \cite{Segev-Rodnay} \cite{Kanso-et-al.-2007} \cite{Arnold2007} \cite{Yavari2008} \cite{Yamaoka_2010}, \cite{Stefanov1996} from the general conservation law we denote velocity by $u=\partial_t\nu$, where displacement $\nu$ is a vector-valued $0$-form. In other words, $\nu$ is locally a vector of displacements in as many directions as the dimension of $\Omega_s$. The displacement is a diffeomorphism of the material particles (or material points) from a reference configuration to a deformed body \cite{Kanso-et-al.-2007}. Especially, this map exists without the metric structure.

When expressed without metric, stress $\sigma$ is a covector-valued $2$-form. Informally, as explained in \cite[chp 38]{Feynman}, stress boils down to the idea of  "force per unit area". Stress maps $2$-vectors (or ordered pairs of ordinary $1$-vectors) representing "virtual oriented areas" to a co-vector that corresponds with virtual work. The virtual work is the metric-independent counterpart to the force \cite{Bossavit-Japan2} interpreted as the "work per length". (This interpretation requires the metric structure.) 

Following similar kind of of reasoning, linearised strain $\varepsilon$ becomes a vector-valued $1$-form. As explained in \cite[chp 38]{Feynman}, an informal interpretation of strain boils down to the idea of  "stretch per unit length". Consequently, in a more general sense, the idea of the metric-independent linearised $\varepsilon$ is to map ordinary $1$-vectors to displacement $\nu$. The source term is the body force $f_V$, which is a vector valued $3$-form.

The stress-strain relation is established with the Hodge operator as a map from vector-valued $1$-forms to covector-valued 2-forms \cite{Yavari2008} \cite{Kovanen2016}. Informally, this is the map between the "stretches per unit length" and the "forces per unit areas". Here, we denote such a Hodge operator by  $\star_s^C$, where the superscript  $C$ is employed to denote all the parameters of the stress-strain relation are incorporated into the Hodge operator. 

Notice, that strain is of the form $\varepsilon=s\otimes f$, and in dimension $n=3$ both $f$ and $s$ have three components. Consequently, strain $\varepsilon$ involves $3\times 3=9$ elements. Stress $\sigma$ has similarly nine elements, and hence, the $\star_s^C$ operator contains $9\times9=81$ elements (which can be further reduced by symmetry considerations). Notice also, the product $\varepsilon\wedge{\star_s^C}\varepsilon = \varepsilon\wedge\sigma$ yields an ${\rm End}(E)$-valued form whose degree is $n$. When the trace of the product is integrated over space, one gets the strain energy \cite{Arnold2007}
\begin{displaymath}
\frac{1}{2}\int\limits_{\Omega_s} {\rm tr}(\varepsilon\wedge{\star_s^C}\varepsilon) = \frac{1}{2}\int\limits_{\Omega_s} {\rm tr}(\varepsilon\wedge\sigma).
\end{displaymath}

Now, for the conservation law we set $F^0=u$, $f^1_s = \varepsilon$, $g^0=-\star_s f_v$ and denote mass density by $\rho$. When these are substituted back to corollary \ref{crl:Conservation law in space and time II} and the constitutive relations are taken into account one gets
\begin{alignat*}{3}
\partial_t\varepsilon \, -\, {\rm d}_\nabla^s u \, & = \, 0\quad &&\text{3rd row}\, , \\
{\rm d}_\nabla^s\varepsilon \,& = \, 0\quad &&\text{6th row}\, , \\
{\star_s{{\rm d}_\nabla^s \star_s^C}}\varepsilon \, -\, {\star_s{\partial_t\star_s^\rho}} u\, & = \, -{\star_s f_v}\quad &&\text{8th row}\, .
\end{alignat*}
Since we have $u=\partial_t\nu$, the equation from the third row is equivalent to $\varepsilon={{\rm d}_\nabla^s} \nu$. Thereafter the second
equation following from the sixth row becomes a tautology. The last equation is equivalent to ${\rm d}_\nabla^s\sigma - {\star_s^\rho \partial_t u} =-f_v$.

Summing up, the conservation law for small-strain elasticity can be written as
\begin{alignat*}{3}
-\partial_t {\varepsilon} \, +\, {{\rm d}_\nabla^s} u \, & =\, 0\,,\quad& \sigma \, &=\, {\star_s^C \varepsilon} \, ,\\ 
{\star_s^\rho \partial_t u} \, -\, {\rm d}_\nabla^s \sigma \, & =\,  f_v\, , \quad &u\, &=\, \partial_t\nu\, ,
\end{alignat*}

The textbook counterpart of the conservation law is obtained from remark \ref{rmk:metric proxy} by choosing $\psi=\overline{u}$, $\overline{r} = \overline{\varepsilon}$, $\beta=-\overline{f}_v$, $\alpha_r=C$, and $\alpha_\psi=\rho$. (The symbols with an overline indicate vector valued functions.) As a result one gets
\begin{alignat*}{3}
-\partial_t \overline{\varepsilon} \, +\, {\rm grad}\,\overline{u} \, & =\, 0\,,\quad& \overline{\sigma} \, &=\, C\overline{\varepsilon} \, ,\\ 
\rho \partial_t  \overline{u} \, -\, {\rm div}\,\overline{\sigma} \, & =\,  \overline{f}_v\, , \quad &\overline{u}\, &=\, \partial_t\overline{\nu}\, .
\end{alignat*}

\subsection{Yang-Mills equation}

To demonstrate that the general conservation law need not to be restricted only to ordinary and $E$-valued forms, our last example is the Yang-Mills equation \cite{Yang-Mills1954} \cite{Yang2014} \cite{Tao2004}. It calls for ${\rm End}(E)$-valued forms. As explained by Baez and Muniain \cite{Baez-Muniain}, the Yang-Mills equation is structurally an immediate extension of Maxwell's equations. Following their work, the counterparts to the electric and magnetic fields are the ${\rm End}(E)$-valued 1-form $e$ and the ${\rm End}(E)$-valued 2-form  $b$, respectively. The source term is $j-\rho\,{\rm d}t$. (In this case the decomposition of the exterior covariant derivative into time and space-like parts is ${\rm d}_\nabla = {\rm d}t\wedge\nabla_t + {\rm d}^s_\nabla$, see \cite[p. 262]{Baez-Muniain}). The Yang-Mills equation can then be concretised in the same manner as the above Maxwell's equations from the general conservation law. This results in 
\begin{alignat*}{1}
{\rm d}_\nabla^s\, b\, &=\, 0\,, \\
{\rm d}_\nabla^s e \, +\, \nabla_t b\, &=\, 0,\\
-{\star_s{{\nabla_t}}\star_s}\, e\, + \, \star_s {{\rm d}_\nabla^s \star_s}\, b\, &=\, \star_s j,\\
 \star_s {{\rm d}^s_\nabla\star_s}\, e\, &=\, \star_s\rho \, .
\end{alignat*}
For further details, see \cite{Baez-Muniain}.

\section{Approximations in finite dimensional spaces}
\label{sec:Approximations}

The three building boxes of the general conservation law are a pair of field $\{F,F^*\}$, a pair of differential  equations ${\rm d}F=0$, ${\rm d}F^\ast = G$, and the constitutive relation $F^\ast = \star F$ (where ${\rm d}$ is now short for  the covariant exterior derivative and the exterior derivative). The main principle behind the commonly employed numerical techniques is to maintain two of the equations exact in finite dimensional spaces and to approximate the remaining third one. 

For example, the finite element method \cite{Ciarlet} satifies one of the differential equations and the constitutive relation exactly in finite dimensional spaces while the remaining differential equation is approximated in the "weak sense" employing variational techniques that follow from the action principle. 

In Yee-like schemes \cite{Yee1966} \cite{Bossavit-Kettunen1999} --known also as the finite difference method or generalized finite differences \cite{BossavitPIER2001}, discrete exterior calculus \cite{Hirani2003}, finite integration technique \cite{Weiland1984}, etc.-- the strategy is to fulfil the pair of differential equations exactly in finite dimensional spaces, while the constitutive relation is approximated only on a finite subset of points of the manifold. The "generalized view" of finite differences \cite{Yee1966} \cite{Bossavit-Kettunen1999} \cite{BossavitPIER2001} \cite{Hirani2003} emphasizes the importance of recognizing the differential equations call only for a differentiable structure. For, this then implies the metric structure is only needed to express the constitutive relations. This is to say, the approximation of the Hodge operator in finite dimensional spaces, that is, "the discrete Hodge" becomes of especial interest \cite{Tarhasaari-Kettunen-Bossavit1999}. (Notice, however, the "discrete Hodge" is not a Hodge operator. The term "discrete Hodge" is just a compound word.) Examples of applying Yee-like schemes to the general conservation law can be found in \cite{RabinaESAIM2018}.

\section{Conclusions}
\label{sec:conclusions}

In this paper we have introduced a general field and its conservation.  The general field is a forma sum of "all" differential forms that may 
appear in field theories, and the conservation law yields solutions whose small changes do not change an action to the first order.  This expresses the possibilities regarding differentiating general fields. The approach is built such that that particular differential equations involved in physics field theories become instances of the general conservation law. This has several advantages in developing software systems to solve for physics boundary value problems. A single system can be implemented to cover a large class of problems without restricting the set of eligible problems to a fixed list given a priori. Furthermore, the general conservation law can be employed to test the correctness of specific models by checking whether they can be instantiated from the general model.

\section*{References}
\bibliographystyle{amsplain}
\bibliography{references}

\end{document}